\documentclass{amsart}

\usepackage{amsmath, amsthm,  amsfonts, amssymb}

\usepackage{mathrsfs} 
\usepackage{array}
\usepackage{hyperref}
\usepackage[all]{xy}
\usepackage{slashbox}
\input xy 

\xyoption{all}
\xyoption{2cell} 
\xyoption{v2}

\DeclareMathOperator{\vol}{vol}
\DeclareMathOperator{\Aut}{Aut}
\DeclareMathOperator{\Lie}{Lie}

\DeclareMathOperator{\id}{id}

\DeclareMathOperator{\Map}{Map}

\DeclareMathOperator{\ad}{ad}

\DeclareMathOperator{\Diff}{Diff}

\DeclareMathOperator{\Eq}{Eq}
\DeclareMathOperator{\Split}{Split}
\theoremstyle{plain}
\newtheorem{theorem}{Theorem}[section]
\newtheorem{corollary}[theorem]{Corollary}
\newtheorem{lemma}[theorem]{Lemma}
\newtheorem{proposition}[theorem]{Proposition}

\theoremstyle{definition}
\newtheorem{definition}[theorem]{Definition}

\theoremstyle{remark}

\newtheorem{remark}{Remark}[section]

\numberwithin{equation}{section}
\numberwithin{figure}{section}

\newcommand{\cC}{{\mathcal C}}

\newcommand{\cG}{\mathcal{G}}

\newcommand{\cA}{{\mathcal A}}

\newcommand{\RR}{{\mathbb R}}

\renewcommand{\a}{\alpha}
\renewcommand{\b}{\beta}

\newcommand{\fg}{{\mathfrak g}}
\newcommand{\fh}{{\mathfrak h}}
\newcommand{\fk}{{\mathfrak k}}
\newcommand{\fl}{{\mathfrak l}}

\newcommand{\wt}{\widetilde}

\newcommand{\im}{\text{im\ }}

\newcommand{\<}{\langle}
\renewcommand{\>}{\rangle}

\begin{document}

\title{The general caloron correspondence}

 \author[P. Hekmati]{Pedram Hekmati}
  \address[P. Hekmati]
  {School of Mathematical Sciences\\
  University of Adelaide\\
  Adelaide, SA 5005 \\
  Australia}
  \email{pedram.hekmati@adelaide.edu.au}

 \author[M. K. Murray]{Michael K. Murray}
  \address[M. K. Murray]
  {School of Mathematical Sciences\\
  University of Adelaide\\
  Adelaide, SA 5005 \\
  Australia}
  \email{michael.murray@adelaide.edu.au}

 \author[R. F. Vozzo]{Raymond F. Vozzo}
  \address[R. F. Vozzo]
  {School of Mathematical Sciences\\
  University of Adelaide\\
  Adelaide, SA 5005 \\
  Australia}
  \email{raymond.vozzo@adelaide.edu.au}

\date{\today}

\thanks{The authors acknowledge the support of the Australian Research Council.}


\begin{abstract}
We outline in detail the general caloron correspondence for the group of automorphisms of  an arbitrary
principal $G$-bundle $Q$ over a manifold $X$, including the case of the gauge group of  $Q$.  These results are used to define characteristic classes of gauge group bundles.  Explicit but complicated differential form representatives are computed in terms of a connection and Higgs field.

\end{abstract}

\maketitle

\tableofcontents

\section{Introduction}

The caloron correspondence arose originally \cite{Garland:1988} as a correspondence between calorons (instantons on $\RR^3 \times S^1$) with structure group $G$ and Bogomolny monopoles on $\RR^3$ with structure group the loop group of $G$.  Later it was realised that the self-duality and Bogomolny equations can be disregarded and interesting results can be obtained by regarding the caloron correspondence as a correspondence between $G$-bundles with connection on $M \times S^1$, for some manifold $M$ and loop group bundles with connection and Higgs field on $M$. In particular in \cite{Murray:2003}, the caloron correspondence  was used to calculate the string class of an $LG$-bundle, and generalised in \cite{Murray-Vozzo1, Murray-Vozzo2}, where it was used to define characteristic classes for $\Omega G$-bundles and $LG \rtimes S^1$-bundles. See also \cite{Bergman:2005, Bouwknegt:2009, Vozzo:2010} for related applications of the caloron correspondence. 

In the current work we generalise these constructions by replacing the circle $S^1$ by an arbitary compact, connected manifold $X$. In summary, 
if   $Y \to M$ is a fibration with fibre $X$ and $\widetilde P \to Y$ is a $G$-bundle, which over a fibre of $Y \to M$ is isomorphic to some
$G$-bundle $Q \to X$, we show that $\widetilde P \to Y$ is equivalent to an infinite-dimensional principal bundle $P \to M$ whose
structure group is the group $\Aut(Q)$ of automorphisms of the $G$-bundle $Q \to X$ or a subgroup thereof. Which subgroup occurs depends 
on whether the fibration $Y \to M$ is a product and whether the bundles are framed. The resulting four possible cases are detailed 
in Section \ref{sec:description} without proof. 

In Section \ref{sec:construction} we give the detailed proofs of the correspondence 
in the most general cases, leaving some of the specialisations for the reader. In the following section we 
show how the correspondence works when $\widetilde P$ has a connection and $P$ a connection and Higgs field.  The main results of the paper are in these two sections and summarised in Theorems \ref{thm:fib-unframed-conn-higgs} and \ref{thm:fib-unframed-conn-form} 
for the case of a general fibration and Theorems \ref{thm:prod-unframed-conn-higgs} and \ref{thm:prod-unframed-conn-form}  
for the case when the fibration is a product.
In Section \ref{sec:classes} we use these constructions to define characteristic classes of gauge group bundles
and define differential form representatives for them in terms of connections and Higgs fields. Finally in Section \ref{sec:universal} we consider the group of based gauge transformations and provide explicit formulas for its universal characteristic classes.

\section{Description of the  caloron correspondences}
\label{sec:description}

There are four basic caloron correspondences.  To introduce  notation we will discuss here what they are correspondences between and leave the definition of the actual correspondences until later. 

Let us fix  $Q \to X$ a (principal) $G$-bundle. Define $\Aut(Q)$ to be all bundle automorphisms of $Q$. A bundle automorphism defines a diffeomorphism of $X$ and  thus there is a  homomorphism $\Aut(Q) \to \Diff(X) $ whose
kernel we denote by $\cG$ and whose image we denote by $\Diff^Q(X)$.  We therefore have a short exact sequence
\begin{equation}
\label{eq:basicshortexact}
1 \to \cG \to \Aut(Q) \to \Diff^Q(X) \to 1  .
\end{equation}
Note that, unlike the case where $X = S^1$, we may not have $\Diff^Q(X) $ equal to $\Diff(X)$.  To understand how this can happen note that a diffeomorphism $\psi \in \Diff(X)$ is in the image of the map  $\Aut(Q) \to \Diff(X)$ if and only if $\psi^*(Q) \simeq Q$.  For example, if $Q \to S^2$ is the standard Hopf bundle of Chern class $1$ and $\psi$ is 
the antipodal map then $\psi^*(Q)$ has Chern class $-1$, so it cannot be isomorphic to $Q$.  
It is a standard fact that if   $\psi$ and $\chi$ are homotopic then $\psi^*(Q) \simeq \chi^*(Q)$, which shows that  the connected component of the identity $\Diff(X)_0$ is a subgroup of $\Diff^Q(X)$.  

Notice that if $R \to M$ is an $\Aut(Q)$-bundle it induces an associated bundle with fibre $X$ defined by $R \times_{\Aut(Q)} X$, where $\Aut(Q)$ acts on $X$ via the homomorphism to $\Diff^Q(X)$ above.  If  $Y \to M$ 
is a locally trivial fibre bundle with fibre $X$  and  structure group 
$\Diff^Q(X)$  denote by $F(Y)\to M$ its $\Diff^Q(X)$ frame bundle. If  $m \in M$ and $f \in F_m(Y)$ then, by definition,  $f \colon X \to Y$ is a diffeomorphism onto the fibre $Y_m$ which we call a {\em frame} at $m$.  Notice that not all 
diffeomorphisms $X \simeq Y_m$ are frames unless $\Diff^Q(X) = \Diff(X)$. 

\begin{definition}
\label{def:type}
 Let $Y \to M$ be a locally trivial fibre bundle with fibre $X$ and structure group $\Diff^Q(X)$ and $\widetilde P \to Y$ 
 be  a principal $G$-bundle. We say that $\widetilde P$ has {\em type} $Q \to X$ if for all $m \in M$ and for all $f \in F(Y)$ we 
 have $f^*(\widetilde P) \simeq Q$. 
 
 \end{definition}

\begin{remark}
Notice that if $f, \tilde f \in F(Y)_m$, then $f = \tilde f \circ \chi$ for $\chi \in \Diff^Q(X)$ so that $f^*(Q) \simeq {\tilde f}^*(Q)$.  Moreover if $M$ is connected  and $m' \in M$, we can join $m$ and $m'$ by a path which can be lifted to $F(Y)$
as a path joining $f \in F(Y)_m $ and some $f' \in F(Y)_{m'}$ so that $f^*(Q) \simeq (f')^*(Q)$. 
Hence for a connected manifold $M$ it suffices to check the type condition at a single point $m \in M$ for a single framing of $Y_m$.
\end{remark}

We have the following correspondences:

\bigskip
\noindent{\sl 1. The unframed caloron correspondence for fibrations.}

There is a bijection between isomorphism classes as follows.
\begin{itemize}
\item $G$-bundles $\widetilde P \to Y$ of type $Q \to X$  and;
\item  $\Aut(Q)$-bundles $P \to M $ with an isomorphism of  spaces over $M$ from  $P \times_{\Aut(Q)} X \to M $ to $Y \to M$. 
\end{itemize}

\bigskip
\noindent{\sl 2. The  unframed caloron correspondence for products.}

If the fibration is a product $Y = M \times X$ then this becomes   a bijection between isomorphism classes as follows.

\begin{itemize}
\item $G$-bundles $\widetilde P \to M \times X$ of type $Q \to X$  and;
\item  $\cG$-bundles $P \to M$. \bigskip
\end{itemize}

We will call these two correspondences the {\em unframed case} to distinguish them from the next two cases. 
First we make a general definition.  

\begin{definition} 
We say a fibre bundle $W \to Z$ is {\em framed} (over a submanifold  $Z_0 \subset Z$) if we have
chosen a section $s \colon Z_0 \to W$. We call $s$ a framing (over $Z_0$).
\end{definition}

We need a number of cases of this definition.  Firstly for the $G$-bundle $Q \to X$ we can pick $x_0 \in X$ and $q_0 \in Q_{x_0}$.
This amounts to a framing of $Q$ over $\{x_0 \}$. We often call $x_0$ and $q_0$ basepoints for $X$ and $Q$.  Secondly
for the fibration $Y \to M$ a framing over $M$ is simply called a framing. Thirdly if $s \colon M \to Y$ is a framing,
we will be interested in $G$-bundles framed over $s(M) \subset Y$. Again we will call these just framed.  

In 
each case there is a natural notion of morphism that preserves the framing. 
Consider then  $x \in X$ and $q \in Q_x$. By restriction the short exact sequence \eqref{eq:basicshortexact}
to framed isomorphisms we have a short exact sequence
\begin{equation}
\label{eq:fbasicshortexact}
1 \to \cG_0 \to \Aut_0(Q) \to \Diff^Q_0(X) \to 1
\end{equation}

Note that, because $Y \to M$ is not a principal bundle, a global section does not necessarily make it trivial. However, in the case that it \emph{is} trivial we take the framing of $M \times X \to M$ to be $m \mapsto (m, x)$ where $x$ is the basepoint for $X$.    
Again we have a notion of the {\em framed type} of a framed bundle by restricting all morphisms in Definition \ref{def:type} 
to be framed. Notice that if $R \to M$ is an $\Aut_0(Q)$-bundle it induces an associated framed bundle with fibre $X$ defined by $R \times_{\Aut_0(Q)} X$, where as above $\Aut_0(Q)$ acts on $X$ via the homomorphism to $\Diff_0^Q(X)$. 

With these definitions we have the following correspondences.

\bigskip
\noindent{\sl 3. The framed caloron correspondence for fibrations.}

There is  a bijection between isomorphism classes as follows.
\begin{itemize}
\item Framed $G$-bundles $\widetilde P \to Y$ over a framed fibration $Y \to M$ of framed type $Q \to X$  and;
\item  $\Aut_0(Q)$-bundles $P \to M $ with a framed isomorphism of framed spaces over $M$ from  $P \times_{\Aut_0(Q)} X \to M $ to $Y \to M$. 
\end{itemize}

\bigskip
\noindent{\sl 4. The framed caloron correspondence for products.}

If the fibration is a product $Y = M \times X$ then this   becomes  a bijection between isomorphism classes as follows.
\begin{itemize}
\item Framed $G$-bundles $\widetilde P \to M \times X$ of type $Q \to X$  and;
\item  $\cG_0$-bundles $P \to M$. 
\end{itemize}

\section{Construction of the caloron correspondences }
\label{sec:construction}
We note first from \cite{Cur} that the groups in equation (\ref{eq:basicshortexact}) are Fr\'echet Lie groups
with $\Aut(Q)$ a Fr\'echet Lie subgroup of $\Diff(Q)$ and $\Diff^Q(X)$ an open subgroup of $\Diff(X)$.  We will assume throughout that all infinite-dimensional spaces are Fr\'echet manifolds.  For details see for example \cite{Ham}. 

To establish the caloron correspondence we need the following useful fact. 
\subsection{Principal bundles and extensions}
\label{sec:pb-extensions}
Let 
$$
1 \to L \xrightarrow{\alpha} H \xrightarrow{\beta} K \to 1
$$
be an extension of Fr\'echet Lie groups, that is a short exact sequence of groups such that
$\a$ is an immersion and $\b$ a submersion in the Fr\'echet sense.  In particular  $\beta$ admits local sections and 
thus $H \to K$ is a locally trivial $L$-bundle.  Indeed if $s \colon U \to H$ is a local
section of $H \to K$ for $U \subset K$, then we can define $U \times L \to H $ by 
$(k, l ) \mapsto s(k)l$ with the inverse $h \mapsto (\beta(h), s(\beta(h))^{-1}h)$. 

We are interested in the relationships between the principal bundles in the 
following diagrams
\begin{equation} 
\xy 
(-10,10)*+{S}="1"; 
(0,0)*+{S/L}="2"; 
(-10,-10)*+{M}="3"; 
(-14,0)*+{H}="4";
(-2,7)*+{L}="5";
(-2,-7)*+{K}="6";
(30,10)*+{T}="7"; 
(40,0)*+{R}="8"; 
(30,-10)*+{M}="9"; 
(38,7)*+{L}="11";
(38,-7)*+{K}="12";
{\ar "1";"2"}; 
{\ar "1";"3"}; 
{\ar   "2";"3"};
{\ar "7";"8"}; 
{\ar   "8";"9"}
\endxy
\end{equation}

If $S \to M $ is an $H$-bundle, then $L$ acts on $S$ on the right and the 
set of orbits $S/L$ is a principal $H/L = K$-bundle over $M$.  Moreover $S \to S/L$ is a principal $L$-bundle. 
Notice that $H$ acts on $S/L$ on the right and the action on $S$ covers this. 

As $L$ is normal in $H$ the adjoint action of $H$ on itself fixes $L$. Given a  principal $L$-bundle $T \to R$ with 
$H$ acting on $R$, we say that $T$ is an {\em $H$-equivariant bundle} if the $L$-action on $T$ can be extended to an 
$H$-action on $T$ which moreover covers the $H$-action on $R$.   In the case above we clearly have that $S \to S/L$ is an $H$-equivariant bundle.
We have

\begin{proposition}
\label{prop:bundles-extensions}
Fix a principal $K$-bundle $R \to M$. We  have a bijective correspondence between isomorphism classes of the following objects:
\begin{enumerate}
\item Principal $H$-bundles $S \to M$ with $S/L \to M$ isomorphic to $R \to M$ as $K$-bundles; and
\item Principal $L$-bundles $T \to R$ which are $H$-equivariant for the 
$H$-action on $R$ induced by the $K$-action on $R$ using the homomorphism $\beta$.
\end{enumerate}
\end{proposition}
\begin{proof}
The correspondence is as described above except that we need to check local triviality. In the forward  direction 
if $S \to M$ is locally trivial we can cover $M$ by open sets $U$ so that $S_U \simeq U \times H$ and thus
$(S/L)_U \simeq U \times H/L$ so that $S_U \to R_U $ is isomorphic as a principal $L$-bundle to $U \times H \to U \times H/L$. As the extension $H \to H/L$ is a locally trivial $L$-bundle the result follows. 

In the backwards direction let $m \in M$ and we want to show that there is an open set $U$ containing $m$ such 
that $T_U \to U$ is a locally trivial $H$-bundle, namely $T_U \simeq U \times H$. First choose $U$ so that $R \to M$ 
is trivial, that is $R_U \simeq U \times K$. As a consequence we have a section $U \to R$ and the pullback of $T \to R$
by that section is locally trivial and thus admits a section in some open neighbourhood of $m$.  So without loss
of generality we can assume that open neighbourhood is $U$ and we have a section $\sigma \colon U \times \{ 1 \} \to T$
of $T $ restricted to $U \times \{1 \}$.  We define
\begin{align*}
\psi \colon U \times H &\to T_U \\
(x, h) &\mapsto \sigma(x) h .
\end{align*}
This is a smooth bijection of right $H$ spaces. It suffices to show that its inverse is also smooth. We do this by covering
$T_U$ with open sets on which the inverse is manifestly smooth. Let $p_0 \in P_U$ with $\pi(p_0) = (m_0, k_0) \in U \times K$. Choose an open neighbourhood $V$ of $k_0$ in $K$ with a local section $s \colon V \to H$ of $\beta$.  Because $T \to R $ is a locally trivial $L$-bundle,
the map $\tau \colon T \times_R T \to L$ defined by $t_1 \tau(t_1, t_2) = t_2$ is smooth. The restriction $\psi^{-1}$  to $V$ is the smooth map   $ V \to U \times \beta^{-1}(V)$ given by $p \mapsto (\pi_U(p), s(\pi_K(p))\tau(\sigma(\pi_K(p))s(\pi_K(p)), p)$, where $\pi_U$ and $\pi_K$ are the natural projections onto $U$ and $K$ respectively. We have established the desired local 
triviality.

\end{proof}

\subsection{The unframed caloron correspondence for fibrations}
\label{sec:unframed-fibrations}
We wish to define the correspondences which establish  the bijection between  isomorphism classes of the
following objects.
\begin{itemize}
\item $G$-bundles $\widetilde P \to Y$ of type $Q \to X$  and;
\item  $\Aut(Q)$-bundles $P \to M $ with an isomorphism of  spaces over $M$ from  $P \times_{\Aut(Q)} X \to M $ to $Y \to M$. 
\end{itemize}

We start with $\widetilde  P \to Y$  a $G$-bundle of type $Q \to X$.   If $Z$ is a space on which $G$ acts on the right (possibly trivially) denote by 
$\Eq	_G(Q, Z)$ the space of all $G$-equivariant maps.  If $G$ acts trivially on $Z$ then $\Eq(Q, Z) = \Map(X, Z)$. 
Thinking of $\Eq(Q,\ )$ as a functor apply it to $\widetilde  P \to Y$. Denote by $\Map^Q(X, Y) \subset \Map(X, Y)$
the image of $\Eq(Q, \widetilde P)$ under the map $\Eq(Q, \widetilde P) \to \Map(X, Y)$. We claim that 
$\Eq(Q, \widetilde P) \to \Map^Q(X, Y)$ is a $\cG$-bundle, noting that $\Eq(Q, G) =\cG$.  To establish local triviality, pick $m \in M$ and 
choose a contractible  open neighbourhood $U$ which can be contracted to $m$.  Because $U$ is contractible it follows that  $Y_U = \pi^{-1}(U) \simeq U \times X$. Moreover we have that the restriction of   $\widetilde P $ to $Y_m \simeq \{m \} \times X$ is  isomorphic to $Q$ and so by contractibility of $U$ it follows that  $\widetilde P_U$, the restriction of $\widetilde P$ to $Y_U$, is diffeomorphic to $U \times Q$.  Then $\Eq(Q, \widetilde P_U) \simeq \Map(X, U) \times \Aut(Q)$ and 
$\Map^Q(X, Y_U) \simeq \Map(X, U) \times \Diff^Q(X)$ and the projection  is the product of the identity on $\Map(X, U)$
with the projection on $\Aut(Q) \to \Diff^Q(X)$.  Local triviality then
follows from the fact that $\Aut(Q) \to \Diff^Q(X)$  is a principal $\cG$-bundle  \cite{Cur}.

Define $\eta \colon F(Y)  \to \Map(X, Y)$ by recalling  that elements of $F_m(Y)$ are maps $X \to Y_m \subset Y$.
Then $\eta^*(\Eq(Q, \widetilde  P) ) \to F(Y)$ is a $\cG$-bundle and $F(Y) \to M$ is a $\Diff^Q(X)$-bundle.  
It follows that there is a projection $\eta^*(\Eq(Q, \widetilde  P) )  \to M$ which we describe as follows. An element of $\eta^*(\Eq(Q, \widetilde P) )$ is a $G$-bundle map $\widehat f\colon Q \to \widetilde P$ covering some frame $f \colon X \to Y$. In fact if $\widetilde P_m$ denotes the restriction of $\widetilde P$ to $Y_m$, then  the fibre of  $\eta^*(\Eq(Q, \widetilde P) )$ above $m \in M$
consists of all $G$-bundle isomorphisms from $Q \to X$ to $\widetilde P \to Y_m$ which cover a frame $X \to Y_m$. 
That is 
$$
(\eta^*(\Eq(Q, \widetilde P) ))_m =  \Eq(Q, \widetilde P_m).
$$

 We can apply 
Proposition \ref{prop:bundles-extensions} to the case of the exact sequence \eqref{eq:basicshortexact} if we can show that the $\cG$-action on $P$ extends to an $\Aut(Q)$-action.  If $\hat \rho \in \Aut(Q)$ is a lift to $\rho \in \Diff^Q(X)$, then we can make $\hat \rho$ act on $\hat f$ by pre-composition and this extends
the $\cG$-action.  We denote
$$
\cC(\widetilde P)  = \eta^*(\Eq(Q, \widetilde P) )
$$
thought of as an $\Aut(Q)$-bundle.  

To see that this is a locally trivial  $\Aut(Q)$-bundle, let us choose again $U \subset M$ contractible so that $Y_U \simeq U \times X$ and $\widetilde P$ restricted to $Y_U$ is diffeomorphic to $U \times Q$. 
 If we consider now the construction just defined, we will see that 
there is a natural isomorphism from $\cC(\widetilde P)$ restricted to $U$ to $U \times \Aut(Q)$ which 
gives a local trivialisation. 

To complete the correspondence as described in the previous Section, we need to describe the isomorphism
$$
\cC(\widetilde P) \times_{\Aut(Q)} X  \to Y
$$
of fibre bundles over $M$.  An element of the first space over $m \in M$ has the form $[(\hat f, f), x]$ where $f \colon X \to Y_m$ is a frame.  We map this to $f(x) \in Y_m$.  The action of $(\hat \rho, \rho)$ is given 
by $((\hat f \circ \hat \rho, f \circ \rho), \rho^{-1}(x) )$ which maps to $f( \rho (\rho^{-1}(x) )) = f(x)$.

Consider now the reverse direction.  We start with an $\Aut(Q)$-bundle $P \to M $ with an isomorphism of  spaces over $M$ from  $P \times_{\Aut(Q)} X \to M $ to $Y \to M$.  Consider the map 
$$
P \times_{\Aut(Q)} Q  \to P \times_{\Aut(Q)} X.
$$
Because $\Aut(Q)$ acts by bundle automorphisms we have a natural action of $G$ on  $P \times_{\Aut(Q)} Q $
by $[p, q] g = [p, qg]$.  It is straightforward to check that this makes 
$$
\cC^{-1}(P) = P \times_{\Aut(Q)} Q  \to P \times_{\Aut(Q)} X  = Y
$$
a $G$-bundle. 

In this case we can assume that locally we have $P \to M$ of the form $U \times \Aut(Q)  \to U$ so that the
bundle
$$
\cC^{-1}(P) = P \times_{\Aut(Q)} Q  \to P \times_{\Aut(Q)} X 
$$
locally looks like
$$
U \times \Aut(Q) \times_{\Aut(Q)} Q  \to U \times \Aut(Q) \times_{\Aut(Q)} X 
$$
or 
$$
U \times Q  \to U \times X.
$$
We can now use the local triviality of $Q \to X$ as a $G$-bundle to  establish that $\cC^{-1}(P) \to Y$ is locally trivial.

The correspondences $\cC$ and $\cC^{-1}$ are best understood as in \cite{Murray-Vozzo1} as functors between
the obvious categories of objects and morphisms although we will not pursue that perspective in the present discussion.  They are not inverses in the set-theoretic sense but only 
in the categorical sense. That is $\cC^{-1} \circ \cC (\widetilde P) \simeq \widetilde P$ and $\cC \circ \cC^{-1} (P) \simeq P$. We construct these isomorphisms next. 

We start with $\widetilde  P \to Y$  a $G$-bundle of type $Q \to X$ and recall that 
$$
\cC(\widetilde P)_m =  \Eq(Q, \widetilde P_m) 
$$
applying $\cC^{-1}$ we have 
$$
\cC^{-1} \circ \cC (\widetilde P) = \Eq(Q, \widetilde P_m)  \times_{\Aut(Q)} Q.
$$
There is a natural map from this space to $\widetilde P_m$ given by $  [\widehat\psi, q] \mapsto \widehat\psi(q)$ which 
defines a smooth isomorphism of $G$-bundles. 

In the other direction let $P \to M$ be an $\Aut(Q)$-bundle. Then
$$
\cC^{-1}(P) = P \times_{\Aut(Q)} Q
$$
so that 
$$
\cC \circ \cC^{-1}(P)_m  = \Eq(Q,  P_m \times_{\Aut(Q)} Q ).
$$
There is a natural map 
$$
P_m \to \Eq(Q,  P_m \times_{\Aut(Q)} Q )
$$
defined by $p \mapsto (q \mapsto [p, q])$ which  extends to an isomorphism of $\Aut(Q)$-bundles. From the categorical viewpoint both these isomorphisms are natural transformations.

\subsection{The unframed caloron correspondence for products}
In the case that the fibration is a product $Y = M \times X$ the caloron correspondence becomes   a bijection between isomorphism classes as follows.

\begin{itemize}
\item $G$-bundles $\widetilde P \to M \times X$ of type $Q \to X$  and;
\item  $\cG$-bundles $P \to M$. 
\end{itemize}

This can be deduced from the general case as follows.  Firstly as $\cG \subset \Aut(Q)$, we can apply the 
construction as before to $P$ to obtain $\widetilde P = \cC^{-1}(P)$.  But in this case we have
$$
\cC^{-1}(P) = (P \times Q)/\cG \to (P \times X)/\cG = M \times X
$$
as $\cG$ acts trivially on $X$.

Secondly, in the other direction, 
consider the fibre of $\cC(P)_m$ which is 
$$
\Eq(Q, \widetilde P_m).
$$
Because $Y = X \times M$ we can  pick out in here the subset of isomorphisms $Q \to P_m$ which 
cover the obvious inclusion $X \to X \times M$, $x \mapsto (x, m)$. This is naturally acted on 
by $\cG$ and the result is a reduction of the $\Aut(Q)$-bundle to $\cG$ which we denote by $\cC(\widetilde P)$. 
Alternatively note that with the previous construction we obtained a $\cG$-bundle 
$$
\eta^*(\Eq(Q, \widetilde P )) \to F(Y) = X \times \Diff^Q(X)
$$
because $Y = X \times M$.  Then we can pull back this $\cG$-bundle with the obvious section of $X \times \Diff^Q(X)$. 
Finally using the map $\bar\eta \colon M \to \Eq(X, M \times X)$ given by 
$m \mapsto (x \mapsto (m, x))$, we get $\cC(\widetilde P) = \bar\eta^*(\Eq(Q, \widetilde P ))$.

\subsection{The framed caloron correspondence for fibrations}
We sketch briefly the correspondences in the framed case.  We want to show that there is  a bijection between isomorphism classes as follows.
\begin{itemize}
\item Framed $G$-bundles $\widetilde P \to Y$ over a framed fibration $Y \to M$ of framed type $Q \to X$  and;
\item  $\Aut_0(Q)$-bundles $P \to M $ with a framed isomorphism of framed spaces over $M$ from  $P \times_{\Aut_0(Q)} X \to M $ to $Y \to M$. 
\end{itemize}

It suffices to show how to define the framings on the transformed objects.  Consider $\cC(\widetilde P)$ whose
fibre at $m \in M$ is 
$$
\cC(\widetilde P)_m =  \Eq(Q, \widetilde P_m).
$$
But now both sides have basepoints so we can restrict to those isomorphisms preserving the basepoints and call
this 
$$
 \Eq_0(Q, \widetilde P_m)  \subset  \Eq(Q, \widetilde P_m) ,
 $$
 which defines a reduction to $\Aut_0(Q)$.  Consider the isomorphism 
 $P \times_{\Aut_0(Q)} X \to M $ to $Y \to M$ which we have defined above to be 
 $$
 [(\hat \psi, \psi), x] \mapsto \psi(x) .
 $$
The framing of the first space is given by the basepoint $ [(\hat \psi, \psi), x_0] $ and $\psi$ preserves
basepoints so that $ [(\hat \psi, \psi), x_0] $ maps to $\psi(x_0)$ which must be the basepoint. 

In the other direction the construction of $\cC^{-1}(P)$ is 
$$
\cC^{-1}(P) = P \times_{\Aut_0(Q)} Q  \to P \times_{\Aut_0(Q)} X  \simeq Y
$$
and both $P \times_{\Aut_0(Q)} Q$ and $P \times_{\Aut_0(Q)} X$ have natural framings coming
from the basepoints of $Q$ and $P$ and the fact that these are preserved by $\Aut_0(Q)$. 

We leave it as an exercise to show that the isomorphisms $\cC \circ \cC^{-1} (P) \simeq P$ and 
$\cC^{-1} \circ \cC (\widetilde P) \simeq \widetilde P$ preserve the framings just defined.

\subsection{The framed caloron correspondence for products}

If the fibration is a product $Y = M \times X$ then this becomes  becomes  a bijection between isomorphism classes as follows.
\begin{itemize}
\item Framed $G$-bundles $\widetilde P \to M \times X$ of type $Q \to X$  and;
\item  $\cG_0$-bundles $P \to M$. 
\end{itemize}

We leave this case also as an exercise for the reader.

\section{The caloron correspondence with connections and Higgs fields}

In this section we will extend the various caloron correspondences from the previous sections to include the data of connections on the bundles involved. 

\subsection{Introduction}
\label{sec:conn-intro}
Before considering how to extend the caloron correspondence to a correspondence for bundles with connections,
we need to recall some facts about principal bundles and to introduce some notation.  If $\pi \colon P \to M$ is a principal $L$-bundle we denote the right action of $l \in L$ 
on $P$ by $R_l \colon P \to P$ and the induced action on forms and tangent vectors  $R_l^*$ and $(R_l)_*$ respectively.   If $\lambda \in \fl = T_e L$ the Lie algebra of $L$, then
$\lambda$ defines a so-called fundamental vector field at any $p \in P$ which we denote by $\iota_p(\lambda)$. Writing $t_0(t \mapsto \gamma(t))$ for the tangent to the map $\gamma$ at $t=0$, we have  
$$
\iota_p(\lambda) = t_0 ( t \mapsto p \exp(t\lambda)) \in T_pP.
$$
If $l \in L$ then $(R_l)_*(\iota_{p}(\lambda)) = t_0 ( t \mapsto p \exp(t\lambda) l) = t_0 ( t \mapsto pl (l^{-1})\exp(t\lambda) l)
=\iota_{pl}(\ad({l^{-1}})(\lambda))$.  A connection one-form $\omega$ on $P$ is an $\fl$-valued one-form on $P$ satisfying
$ \omega_p (\iota_p (\lambda)) = \lambda$ and  $R_l^*(\omega) = \ad(l^{-1}) \omega$.

If $L$ is a subgroup of $H$ which also acts freely on  the right of $P$, extending the action of $L$, the connection $\omega$ is called $H$-invariant 
if it is fixed by $H$. A straightforward calculation shows the following Lemma.
\begin{lemma}
The connection $\omega$ is $H$-invariant if and only if 
$$
R_h^*(\omega) = \ad(h^{-1}) \omega.
$$
\end{lemma}

Consider a fibration $\pi \colon Y \to M$ with fibre $X$ and $F(Y)$ its associated frame bundle which is a principal $\Diff^Q(X)$-bundle. We denote by $T^v_yY$ the vertical tangent vectors or the tangents to the fibres of $\pi$ at $y \in Y$. A connection  $a$ on $Y$ is a complementary subspace to $T^v_y$ at every $y$ or, equivalently a projection $v_a \colon T_y Y \to T^v_yY$. A choice of connection on $Y \to M$ defines a connection on $F(Y)$ as follows. The tangent space to $f \in F(Y)$
is the subspace of $\Gamma(X,  f^*(TY))$ of vectors whose projection to $T_{\pi(f)} M$ is constant.  The Lie algebra 
of $\Diff^Q(X)$ is $\Gamma(X, TX)$.  If $\xi \in \Gamma(X,  f^*(TY))$ is a tangent vector then $v_a(\xi(x))$ is 
a vertical vector at $f(x)$. We have $f \colon X \to Y_{\pi(f(x))}$ a diffeomorphism so we can define 
$a_f(\xi) \in \Gamma(X, TX)$ by 
$$
a_f(\xi)(x) = (f^{-1})_*(v_a(\xi(x))).
$$
Equivalently $\xi$ is horizontal  if and only if $\xi(x)$ is horizontal for all $x \in X$.

\subsection{Principal bundles with connection and extensions}

First we need to reconsider the discussion from Section \ref{sec:pb-extensions} taking into account connections on the bundles in question.

  We have the same structure as before. An exact sequence of groups
  $$
1 \to L \xrightarrow{\alpha} H \xrightarrow{\beta} K \to 1  
  $$
and a commutative diagram

  \begin{equation} 
\xy 
(-10,10)*+{S}="1"; 
(0,0)*+{R}="2"; 
(-10,-10)*+{M}="3"; 
(-14,0)*+{H}="4";
(-2,7)*+{L}="5";
(-2,-7)*+{K}="6";
{\ar "1";"2"}; 
{\ar "1";"3"}; 
{\ar   "2";"3"}
\endxy
\end{equation}

If we fix a point $s \in S_m$, the fibre of $S$ above $m \in M$, there is an $H$-equivariant isomorphism

\begin{equation} 
\xy 
(-20,10)*+{H}="1"; 
(-20,-5)*+{K}="2"; 
(-10,10)*+{S_m}="3"; 
(-10,-5)*+{R_m}="4"; 
(10,10)*+{h}="5"; 
(10,-5)*+{k}="6"; 
(30,10)*+{sh}="7"; 
(30,-5)*+{\pi(s)\beta(h)}="8"; 
{\ar "1";"2"};
{\ar "3"; "4"};
{\ar "1"; "3"};
{\ar "2"; "4"};
{\ar@{|-{>}} "5";"6"};
{\ar@{|-{>}} "7"; "8"};
{\ar@{|-{>}}  "5"; "7"};
{\ar@{|-{>}} "6"; "8"}
\endxy
\end{equation}
determined by the choice of $s$. Here $\pi$ is the projection $S \to R$. The $L$-bundle $H \to K$ has two $H$-actions given by left and right multiplication. We are interested in the right
action which does not commute with the $L$-action but extends it.  An $H$-invariant connection for this 
action is determined by its value at the identity in $H$ which is a splitting of the exact sequence 
of Lie algebras
\begin{equation}
\label{eq:Lie-ghf}
0 \to   \fl   \xrightarrow{\a }\fh \xrightarrow{\b} \fk \to 0.
\end{equation}
Let $\Split(\fh, \fk)$ denote the affine space of all right splittings of \eqref{eq:Lie-ghf}.  We define a left action of $H$ on $ \sigma \in \Split(\fh, \fk)$ by 
$h \sigma  = \ad(h) \sigma \ad({\beta(h)^{-1}})$. 
To extend Proposition \ref{prop:bundles-extensions} we first need the following Definition.

\begin{definition}\label{D:Higgs}
If $S \to M$ is an $H$-bundle, a {\em Higgs field} is a section $\Phi$ of the associated
bundle $S \times_H \Split(\fh, \fk)$ or equivalently a function $\Phi \colon S  \to \Split(\fh, \fk)$
satisfying
$$
  \Phi(sh) = \ad(h^{-1})  \Phi(s)   \ad(\b(h)) .
$$
\end{definition}

Recall that a right splitting of the  exact sequence \eqref{eq:Lie-ghf} gives rise to a {\em left splitting} and 
vice-versa.  Occasionally we will need to distinguish between these and we will use the notation $\Phi^r$ and $\Phi^l$ for right and left splittings respectively.  They are related by
\begin{equation}
\label{eq:leftright}
\Phi^r \beta + \alpha \Phi^l = \id_{\fh}.
\end{equation}
For the moment it is most natural to use right splittings for the value of the Higgs field but we warn the reader that
from Proposition \ref{P:higgs-equivariance} onwards we will be assuming the Higgs field is a left splitting. We will also adopt the convention that $\alpha$ is an inclusion and hence not explicitly referred to. 

\begin{proposition}
\label{prop:bundles-extensions-connections}
Fix a principal $K$-bundle $R \to M$ with connection $\omega^K$. We  have a bijective correspondence between isomorphism classes of the following objects:
\begin{enumerate}
\item Principal $H$-bundles $S \to M$, with $S/L \to M$ isomorphic to $R \to M$ as $K$-bundles with connection $\omega^H$, which projects to $\omega^K$ under the isomorphism, and  Higgs field $\Phi$; and
\item Principal $L$-bundles $T \to R$ which are $H$-equivariant for the 
$H$-action on $R$ induced by the $K$-action on $R$ using the homomorphism $\beta$,  with 
connection $\omega^L$ which is $H$-invariant.
\end{enumerate}
\end{proposition}
\begin{proof}
We assume the isomorphisms from Proposition \ref{prop:bundles-extensions} are in place so it is just a question of constructing the connections.  As noted above for convenience we regard $\fl$ as included in $\fh$ and suppress the function $\alpha$.  The connections and Higgs field are then related by the equation
\begin{equation}
\label{eq:connections}
\omega^H = \omega^L + \Phi  (\pi^*\omega^K)
\end{equation}
where we have denoted the projection $S \to R $ by $\pi$ and have used the fact that
$\Phi(s) \colon \fk \to \fh$ for all $s \in S$.    It is straightforward to check 
that $H$-invariance of $\omega^L$ is equivalent to $R_h^* \omega^L = \ad(h^{-1}) \omega^L$.
Notice that this make sense because $\ad(h^{-1})(\fl) = \fl$ and that if  $h \in L$ this is just part of the condition satisfied by any connection.

First we do the forwards direction. Assuming that $\omega^H$ and $\Phi$ are given, we 
show that the one-form $\omega^L$ defined by equation \eqref{eq:connections} is an $H$-invariant connection on $S \to R$.  Let $\lambda \in \fl$ and $s \in S$, then
\begin{align*}
\omega_s^L\big(\iota^L_s(\lambda)\big) &= \omega^H\big(\iota^L_s(\lambda)\big) - \Phi(s)\big( \omega_{\b(s)}^K(\pi_*(\iota^L_s(\lambda)))\big)\\
&=\omega^H\big(\iota^H_s(\a(\lambda))\big) - \Phi(s)\big( \omega_{\b(s)}^K(0)\big)\\
&=\lambda
\end{align*}
as required.  
Letting $h \in H$ and $\xi \in T_s S$ we have
\begin{align*}
\alpha\big( (R_h^* \omega^L_{sh})(\xi)\big) &= \a \big( \omega^L_{sh}((R_h)_*(\xi))\big) \\
 &= \omega^H_{sh}\big((R_h)_*(\xi)\big) - \Phi(sh)\big(\omega^K_{\pi(s)\b(h)}((R_{\b(h)})_*(\pi_*(\xi)))\big) \\
 &= \ad(h^{-1}) \omega^H_s(\xi) - \\
 & \hspace{0.2\textwidth} \ad(h^{-1}) \Phi(s) \ad(\b(h)) 
 \ad(\b(h)^{-1})\omega^K_{\pi(s)}(\pi_*(\xi))\\
 &=\ad(h^{-1}) \big\{ \omega^H_s(\xi)  - \Phi(s)(\pi^*(\omega^K_s(\xi))) \big\}\\
 &= \ad(h^{-1})(\omega^L_s) (\xi).
\end{align*}
Hence $R_h^* \omega^L = \ad(h^{-1}) \omega^L$ and we conclude that $\omega^L$ is an $H$-invariant 
connection on $S \to R$.

Consider the reverse direction.  We are given $\omega^L$ which is $H$-invariant and we wish to 
manufacture $\Phi$ and define $\omega^H$ using equation \eqref{eq:connections}. We define $\Phi$
as follows.  Let $s \in S$ and $\kappa \in \fk$. Consider $\iota^K_{\pi(s)}(\kappa) \in T_{\pi(s)}R$
and lift it to a horizontal vector $\widehat{\iota^K_{\pi(s)}(\kappa)} $ at $s \in S$ using $\omega^L$. The projection of this
vector to $M$ is the projection of $\iota^K_{\pi(s)}(\kappa)$ to $M$ which is zero, so
it must be vertical for $H$.  Hence we can define $\Phi(s) \colon \fk \to \fh$ by 
$$
\iota_s^H(\Phi(s)(\kappa)) = \widehat{\iota^K_{\pi(s)}(\kappa)}.
$$
We need to check that this is a Higgs field. First we check it splits. We have
\begin{align*}
\iota^K_{\pi(s)}(\kappa) &=  \pi_*\big(\widehat{\iota^K_{\pi(s)}(\kappa)}\big) \\
      & = \pi_*\big(\iota_s^H(\Phi(s)(\kappa))\big) \\
      & = \iota_{\pi(s)}^K\big(\beta(\Phi(s)(\kappa))\big),
      \end{align*}
where the last line follows from the fact that the $H$-action on $S$ covers the $H$-action on $R$ induced by $\beta\colon H \to K$. Thus we have
$$
\kappa = \beta\big(\Phi(s)(\kappa)\big) 
$$
as required. Next we check that it transforms correctly:
\begin{align*}
\iota^H_{sh}\big(\ad(h^{-1})\Phi(s)(\kappa)\big) &= (R_h)_*\big(\iota^H_s(\Phi(s)(\kappa))\big)\\
  &= (R_h)_*\big(\widehat{\iota^K_{\pi(s)}(\kappa)}\big)\\
  &= \widehat{(R_{\b(h)})_*\big(\iota^K_{\pi(s)}(\kappa)\big)}\\
  &= \widehat{ \iota^K_{\pi(sh)}\big(\ad(\b(h)^{-1})(\kappa)\big)}\\
  &= \iota_{sh}^H\Big( \Phi(sh)\big((\ad(\b(h)^{-1})(\kappa))\big)\Big)
\end{align*}
hence
$$
  \Phi(sh) = \ad(h^{-1})  \Phi(s) \ad(\b(h)).
$$
It remains to show that  $\omega^H$ defined by equation \eqref{eq:connections} is a connection. First we show that $\omega_s^H(\iota_s^H(\eta)) = \eta$ for all $\eta \in \fh$.  For such a $\eta$ we have
$$
\eta = \zeta + \Phi(s)(\beta(\eta))
$$
for some $\lambda \in \fl$. Again we suppress $\alpha \colon \fl  \to \fh$. It follows that
\begin{align*}
\iota^H_s(\eta) &= \iota^H_s(\lambda) + \iota^H_s(\Phi(s) (\beta(\eta))) \\
      &= \iota^L_s(\lambda) + \widehat{\iota^K_{\pi(s)}(\beta(\eta))}
\end{align*}
and thus we have 
\begin{align*}
\omega_s^H( \iota^H_s(\eta) ) &= \omega_s^L( \iota^H_s(\eta) ) + \Phi(s)(\omega^K_{\pi(s)}\pi_*(\iota^H_s(\eta)))\\
&= \omega_s^L (\iota^L_s(\lambda)) +  \omega_s^L(\widehat{\iota^K_{\pi(s)}(\beta(\eta))}) + \Phi(s)(\omega^K_{\pi(s)}\pi_*(\iota^L_s(\lambda)) \\
&\hspace{0.4\textwidth} + \Phi(s)(\omega^K_{\pi(s)}\pi_* (\widehat{\iota^K_{\pi(s)}(\beta(\eta))})\\
&=\omega_s^L( \iota^L_s(\lambda) ) + 0 + 0 + \Phi(s)(\omega^K_{\pi(s)}(\iota^K_{\pi(s)}(\beta(\eta)))\\
&= \lambda + \Phi(s) ( \beta(\eta))\\
&= \eta.
\end{align*}
Next we establish the right equivariance of the connection form, $R_h^*\omega^H_{sh} = \ad(h^{-1}) \omega^H_s $. Let $\xi \in T_s S$ and write 
$$
\xi = \hat \xi + \iota_s(\lambda)
$$
where $\hat \xi$ is horizontal for $\omega^L$ and $\lambda \in \fl$.  Then 
$$
(R_h)_*(\xi) =  (R_h)_*(\hat \xi) + \iota_{sh}(\ad(h^{-1})(\lambda))
$$
and $(R_h)_*(\hat \xi)$ is horizontal as $\omega^L$ is $H$-invariant.  We have

\begin{align*}
(R_h^*\omega^H_{sh})(\xi) &= \omega^H_{sh}\left((R_h)_*(\xi)\right) \\
         &= \omega^L_{sh}\left((R_h)_*(\xi)\right) + \Phi(sh)\big(\omega^K_{\pi(s)\beta(h)}(\pi_*((R_h)_*(\xi)))\big)\\
         &= \omega^L_{sh}\big((R_h)_*(\hat \xi)\big) + \omega^L_{sh}\big(\iota_{sh}(\ad(h^{-1})(\lambda))\big) + \Phi(sh)\big(\omega^K_{\pi(s)\beta(h)}(\pi_*((R_h)_*(\xi)))\big)\\
         &= 0 + \omega^L_{sh}\left((R_h)_*\iota_{s}(\lambda))\right) + \Phi(sh)\big(\omega^K_{\pi(s)\beta(h)}((R_{\b(h)})_* \pi_*(\xi))\big)\\
         &= \ad(h^{-1}) \omega^L_s(\lambda) + \ad(h^{-1}) \Phi(s)\big( \ad(\b(h)) \ad(\b(h)^{-1}) \omega^K_{\pi(s)}(\pi_*(\xi))\big)\\
         &= \ad(h^{-1}) \omega^L_s(\xi)  + \ad(h^{-1}) \Phi(s)\big(  \omega^K_{\pi(s)}(\pi_*(\xi))\big)\\
         &= \ad(h^{-1})\omega^H_s (\xi)
\end{align*}
as required.
         
We leave it as an exercise for the reader to show that when we apply the bijection twice the isomorphisms 
introduced in Proposition \ref{prop:bundles-extensions} preserve the connections defined here.
\end{proof}

For the caloron correspondence we need the following slightly more complicated version of Proposition \ref{prop:bundles-extensions-connections}. Choose an affine subspace  $\Split(\fh, \fk)_0 \subset \Split(\fh, \fk)$ which is invariant under the action of $H$. Using the same 
notation as before we say a connection $\omega^L$ is a $\Split(\fh, \fk)_0$-connection if the corresponding Higgs field takes 
its values in $\Split(\fh, \fk)_0$.  It is then obvious that we have

\begin{corollary}
\label{cor:bundles-extensions-connections}
Fix a principal $K$-bundle $R \to M$ with connection $\omega^K$. We  have a bijective correspondence between isomorphism classes of the following objects:
\begin{enumerate}
\item Principal $H$-bundles $S \to M$, with $S/L \to M$ isomorphic to $R \to M$ as $K$-bundles with connection $\omega^H$, which projects to $\omega^K$ under the isomorphism, and  Higgs field $\Phi \colon S \to \Split(\fh, \fk)_0$; and
\item Principal $L$-bundles $T \to R$ which are $H$-equivariant for the 
$H$-action on $R$ induced by the $K$-action on $R$ using the homomorphism $\beta$,  with 
a $\Split(\fh, \fk)_0$-connection $\omega^L$ which is $H$-invariant.
\end{enumerate}
\end{corollary}

\subsection{The unframed caloron correspondence for fibrations with connections and Higgs fields}

Let  $a$ be a connection on $F(Y) \to M$, where 
$F(Y)$ is the principal $\Diff^Q(X)$-bundle associated to $Y$. This induces a connection also on $Y$. 
We  wish to show that there is a  bijection between isomorphism classes of 
\begin{itemize}
\item $G$-bundles $\widetilde P \to Y$ of type $Q \to X$  with connection $\widetilde A$ and connection $a$ for $F(Y) \to M$ and;
\item  $\Aut(Q)$-bundles $P \to M $ with Higgs field $\Phi$  and connection $A$ and isomorphism of $X$ fibrations from  $P \times_{\Aut(Q)} X$ to $Y$ which sends the connection $A$ to $a$. 
\end{itemize}

We will prove this result using the constructions from Section \ref{sec:unframed-fibrations} and Corollary \ref{cor:bundles-extensions-connections}. Recall that in that case the extension of groups is given by \eqref{eq:basicshortexact}
\begin{equation*}
1 \to \cG \to \Aut(Q) \to \Diff^Q(X) \to 1  .
\end{equation*}
The corresponding sequence of Lie algebras is
\begin{equation}
\label{eq:basicshortexactla}
0 \to \Gamma(X, \ad(Q)) \to \Gamma(X, TQ/G) \to \Gamma(X, TX) \to 0  ,
\end{equation}
which is the functor $\Gamma(X, \ )$ applied to the  Atiyah sequence \cite{Ati} of $Q \to X$ given by 
\begin{equation}
\label{eq:atiyahsequence}
0 \to\ad(Q) \to TQ/G \to TX \to 0  .
\end{equation}

\begin{remark}
\label{rem:inclusion}
We note for later use that we can identify $\Gamma(X, \ad(Q)) = \Gamma_G(Q, \fg)$, the space
of $G$-equivariant maps from $Q$ into $\fg$ and $\Gamma(X, TQ/G) = \Gamma_G(Q, TQ)$, the space
of $G$-equivariant vector fields on $Q$. The inclusion map of the former into the latter then maps a 
function $\mu \colon Q \to \fg$ into the vector field $q \mapsto \iota_q(\mu(q))$. 
\end{remark}

Recall that a  connection on $Q \to X$ is a splitting of the Atiyah sequence \eqref{eq:atiyahsequence} with a right 
splitting corresponding to a horizontal distribution and a left splitting to a connection one-form on $Q$. 
A splitting of \eqref{eq:atiyahsequence} also  induces a splitting of \eqref{eq:basicshortexactla}, which 
in turn induces an $\Aut(Q)$-invariant connection on $\Aut(Q) \to \Diff^Q(X)$.  We take as $\Split(\Gamma(X, TQ/G), \Gamma(X, TX))_0$ 
only those splittings arising in this way and we denote them by $\cA$.   It is straightforward to check that $\cA$ is an  affine subspace invariant under $\Aut(Q)$ and we  are in the setting of Corollary \ref{cor:bundles-extensions-connections}. Note also that in this situation a Higgs field is an $\Aut(Q)$-equivariant map from $P$ into splittings of the sequence (\ref{eq:basicshortexactla}). In fact, as we shall see below the construction of the Higgs field from the proof of Proposition \ref{prop:bundles-extensions-connections} adapted to this context naturally takes values in $\Split(\Gamma(X, TQ/G), \Gamma(X, TX))_0 = \cA$. Thus, we can view the Higgs field as an $\Aut(Q)$-equivariant map $P \to \cA$, or equivalently a section of $P \times_{\Aut(Q)} \cA$.

\begin{proposition}
\label{prop:higgs}
Let $f \in \Eq (Q, \widetilde P)$ and $\Phi$ be the Higgs field as constructed in Proposition  \ref{prop:bundles-extensions-connections}. The value of $\Phi(f)$ is the connection $f^{*}(\widetilde A)$. 
\end{proposition}
\begin{proof}
We have $f \colon Q \to \widetilde P_m$ covering $\bar f \colon X \to Y_m$. Following the definition of the Higgs field 
above we take $\mu \in \Gamma(X, TX)$ and choose a one-parameter family of diffeomorphims in $\Diff^Q(X)$, $\chi_t \colon X \to X$ such that $\chi_0 = \id_X$ and $\chi_t'(0) = \mu$. We have
$$
\iota_f(\eta) = t_0( t \mapsto f \circ \chi_t) \in T_f(F(Y)).
$$
The construction of $\Phi(f)$ implies that there is a  one-parameter family of bundle automorphisms $\hat\chi_t \colon Q \to Q$ such that $\hat\chi_0 = \id_Q$, $\hat\chi_t$ covers $\chi_t$ and $\Eq(Q, \widetilde A)_f ((f\circ \hat\chi_t)'(0)) = 0$ which is equivalent to 
 $\widetilde A_{f(q)} ( (f\circ \hat\chi_t(q))'(0)) = 0$ for all $q \in Q$.   The splitting $\Phi( f)$ of \eqref{eq:basicshortexactla} is therefore the 
map $\mu \mapsto {\hat\chi}_t'(0)$.  But because $\widetilde A_{f(q)} ( (f\circ \hat\chi_t(q))'(0)) = 0$ for all $q \in Q$,
we see that this must be a splitting in $\cA$ and one which is associated to the connection one-form $f^{*}(\widetilde A)$.
\end{proof}

We also note that in this situation the equivariance condition on the Higgs field is particularly nice.

\begin{proposition}\label{P:higgs-equivariance}
Let $\Phi$ be a Higgs field for an $\Aut(Q)$-bundle $P$, viewed as a map $P \to \cA$. Then the condition from Definition \ref{D:Higgs} 
is equivalent to
$$
\Phi(p\psi) = \psi^* \Phi(p),
$$
where $\psi \in \Aut(Q)$ and $\beta \colon \Aut(Q) \to \Diff^Q (X)$.
\end{proposition}

\begin{proof}
First note that if we view $\Phi$ as a connection one-form for each $p \in P$, then what we called the Higgs field in Definition \ref{D:Higgs} is really the associated splitting $\Gamma(X, TX) \to \Gamma(X, TQ/G)$. That is, the Higgs field is the assignment of horizontal subspaces of $TQ$ at each point in $Q$. To avoid confusion for this proof we denote this splitting $\Phi^r$ and the connection one-form by $\Phi^l$. They are related by \eqref{eq:leftright}. Then by definition the image of $\Phi^r$ is the kernel of the one-form $\Phi^l$.  

We wish to show that the connection one-form corresponding to the splitting $\ad(\psi^{-1}) \Phi^r \ad (\beta(\psi))$ is $\psi^*\Phi^l$.  It is straightforward to check that $\psi^*\Phi^l$ is a connection one-form so we just need to show that  $\psi^* \Phi^l$ annihilates $\big(\ad(\psi^{-1}) \Phi^r \ad (\beta(\psi))\big)(\xi)$ for $\xi$ a vector field on $X$. The adjoint action of $\Aut(Q)$ on an element $\mu$ in its Lie algebra, $\Lie(\Aut(Q)) = \Gamma(X, TQ/G)$ is given by $\ad(\psi)(\mu)_q = \psi_* \mu_{\psi^{-1}(q)}$. Similarly, if $\xi \in \Lie(\Diff^Q(X))$ then $\ad(\beta(\psi))(\xi)_x = \beta(\psi)_* \xi_{\beta(\psi)^{-1}(x)}.$ Notice that this implies that if $\omega$ is a one-form on $Q$ and $\zeta $ is a vector field on $Q$ then $(\psi^*\omega)_q (\ad (\psi^{-1})\zeta_q) = \omega_{\psi(q)}(\psi_* (\psi^{-1}_{{\phantom-}*} \zeta)_{\psi(q)}) = \omega_{\psi(q)}(\zeta_{\psi(q)})$.  Letting $\zeta = \Phi^H \ad (\beta(\psi))(\xi)$ 
we note that $\Phi(\zeta) = 0$ and then we have 
$$
(\psi^*\Phi^l)\big(\ad(\psi^{-1}) \Phi^r \ad (\beta(\psi))\big)(\xi) =   \Phi^l (\zeta)  = 0
$$
as required.
 \end{proof}
 
We note that from now on the default choice for the value of a Higgs field $\Phi$ will be that it is a connection one-form, that is a left-splitting. When this is not the case we will write $\Phi^r$. 

Now we are in a position to extend the caloron correspondence to include connections and Higgs fields. We start with a $G$-bundle $\widetilde P \to Y$ of type $Q \to X$. We have seen above that $\Eq(Q, \widetilde P) \to \Map^Q(X, Y)$ is a
$\cG$-bundle.  We can apply the functor $\Eq(Q, \ )$ to the connection $\widetilde A$ to obtain a connection $\Eq(Q, \widetilde A)$. 
Indeed if $\rho \in T_f(\Eq(Q, \widetilde P)) $ then $\rho \in \Gamma_G(Q, f^{*}(T\widetilde P))$, so for any $q \in Q$ we have $\rho_q \in T_{f(q)}\widetilde P$ and we 
define
$$
\Eq(Q, \widetilde A)_f(\rho) = (q \mapsto  \widetilde A_{f(q)}(\rho_q)) \in \Gamma(X, \ad(Q)).
$$

\begin{proposition}
In the situation just described $\Eq(Q, \widetilde A) $ is an $\cA$-connection which is $\Aut(Q)$-invariant  on the $\cG$-bundle $\Eq(Q, \widetilde P) \to \Map^Q(X, Y)$. 
\end{proposition}
\begin{proof}
Consider first what happens when we apply  $\Eq(Q, \widetilde A)$ to a vertical vector in the tangent space to $\Eq(Q, \widetilde P)$
at $f$. The tangent space to $f$ is $\Gamma_G(Q, f^{*}(TQ))$ and  a vertical vector is  generated by an element of the Lie algebra of $\cG$ which is given by an equivariant map $\mu \colon Q \to \fg$.  It is a straightforward exercise to check that
$$
\iota_f(\mu)(q) = \iota_{f(q)} \mu(q)
$$
and then 
$$
\Eq(Q, \widetilde A)_f (\iota_f(\mu)) = 
q \mapsto \widetilde A_{f(q)}(\iota_{f(q)} \mu(q) )= \mu(q)
$$
or $\Eq(Q, \widetilde A)_f (\iota_f(\mu)) = \mu$. 

To check invariance under $\Aut(Q)$ we first need to understand the adjoint action of $\Aut(Q)$ on $\cG$ and its Lie algebra. If $g \in \cG$
then $g \colon Q \to G$ and it acts on $Q$ by $q \mapsto q g(q)$.  If $\psi \in \Aut(Q)$ then $\psi^{-1} g \psi$ acts on $Q$ by sending 
$q \in Q$ to 
$\psi^{-1} ( \psi(q) g (\psi(q))) = q g(\psi(q))$ or $\ad(\psi^{-1})(g) = g \circ \psi$. Hence if $\mu \colon Q \to \fg$ is in the Lie algebra of 
$\cG$ we have $\ad(\psi^{-1})(\mu) = \mu \circ \psi$. 
Let $f_t \in \Eq(Q, \widetilde P)$ with $f_0 = f$ and $\rho = t_0(t \mapsto f_t) \in T_f (\Eq(Q, \widetilde P))$. We have 
\begin{align*}
R^*_{\psi} ( \Eq(Q, \widetilde A))_{f} (\rho)(q)  & = \Eq(Q, \widetilde A)_{f \circ \psi} (t_0(t \mapsto f_t \circ \psi)) (q)\\
                     &=  \widetilde A_{f(\psi(q))}\big(t_0(t \mapsto f_t(\psi(q)))\big) \\
                     &= \Eq(Q, \widetilde A)_f (\rho)(\psi(q))\\
                     &= \ad({\psi^{-1}(q)})\Eq(Q, \widetilde A)_f (\rho)
                     \end{align*}
                    as required. Proposition \ref{prop:higgs} implies that the associated Higgs field is valued in $\cA$ and so this is an $\cA$-connection.
\end{proof}

 The pullback of the connection $\Eq(Q, \widetilde A) $ using $\eta \colon F(Y) \to \Map(X, Y)$ is therefore
an $\Aut(Q)$-invariant connection on the $\cG$-bundle $\eta^*(\Eq(Q, P)) \to F(Y)$.    

It follows from Corollary \ref{cor:bundles-extensions-connections}  that we have a connection $A$ on $P \to M$ defined by 
\begin{equation}
\label{eq:tildeconnection}
A = \eta^*(\Eq(Q, \widetilde A)) + \Phi^r(\pi^*a)
\end{equation}
where $\pi \colon P \to F(Y)$. 

Putting this all together with the results from Section \ref{sec:conn-intro} we have the following Proposition.

\begin{theorem}
\label{thm:fib-unframed-conn-higgs}
Let  $\widetilde P \to Y$ be a $G$-bundle  of type $Q \to X$  with connection $\widetilde A$  and let  $a$ be a connection for $Y \to M$. Then the $\Aut(Q)$-bundle $P \to M$ has connection $A$ and Higgs field $\Phi$ defined as follows. 
 \begin{itemize}
\item Let $f \in P$ so that $f \colon Q \to \widetilde P_m$ for some $m \in M$, then 
 $\Phi(f) = f^*(\widetilde A)$;
\item 
 If $\rho \in  \Gamma(Q, f^*(T\widetilde P))$, let $\bar f \colon X \to Y_m$ and $\bar\rho \in \Gamma(X, \bar f^*(TY))$ be the projections of $f$ and $\rho$ respectively.  Then $A_f(\rho) \in \Gamma_G(Q, TQ)$ is given by 
$$
A_f(\rho)(q) = \widetilde A _{f(q)}(\rho(q)) + \Phi^r(f) \big(x \mapsto  ({\bar f}^{-1})_*(v_a(\bar \rho (x)))\big)(q) 
$$
where, as before,  $\Phi^r(f)\colon \Gamma(X, TX) \to \Gamma_G(Q, TQ)$ is the right splitting induced by the connection 
$\Phi(f) = f^{*}(\widetilde A)$. 
\end{itemize}
\end{theorem}

Consider the caloron transform in the other direction.
Given an $\Aut(Q)$-bundle $P \to M$ we define $Y = P \times_{\Aut(Q)} X$ and a $G$-bundle $\widetilde P \to Y$ by 
$\widetilde P  = P \times_{\Aut(Q)} Q$. Recall that the Lie algebra of $\Aut(Q)$ is $\Gamma_G(Q, TQ)$, the Lie algebra of $G$-equivariant vector fields on $Q$.  The tangent space 
at $(p, q) \in P \times Q$ is 
$$
T_{(p, q)}(P \times Q ) = T_p P \times T_q Q
$$
and the vertical vector induced by $\mu \in \Lie( \Aut(Q))$ is given by $(\iota_p(\mu), -\mu(q))$.   

Let $A$ be a connection one-form on $P$. That is, $A$ is an $\Lie( \Aut(Q))$-valued one-form on $P$
so that if $\xi \in T_pP$ then $A_p(\xi) \in \Lie( \Aut(Q)) = \Gamma_G(Q, TQ)$ and thus $A_p(\xi)(q) \in T_q Q$. 
As above a Higgs field for $P \to M$ is an equivariant map $\Phi \colon P \to \cA$, so we have
$$
\Phi(p)_q \left( A_p(\xi)(q)  \right) \in \fg .
$$
We define
\begin{equation}\label{E:widetildeomega}
\widetilde \omega_{(p, q)}(\xi, \zeta) = \Phi(p)_q \left( A_p(\xi)(q) + \zeta \right)
\end{equation}
on the product $P\times Q$ and we have the following Theorem.

\begin{theorem} 
\label{thm:fib-unframed-conn-form}
The one-form 
$$
\widetilde \omega_{(p, q)}(\xi, \zeta) = \Phi(p)_q \left( A_p(\xi)(q) + \zeta \right)
$$
descends to a connection one-form $\widetilde A$ on the $G$-bundle $\widetilde P \to Y$.
\end{theorem}
\begin{proof}
First we show that $\omega$ annihilates vectors generated by the action of $\Aut(Q)$ on $P \times Q$. In the tangent space at $(p, q)$ such vectors have the form $(\iota_p(\mu), - \mu(q))$ for some $\mu \in \Gamma_G(Q, TQ)$ and we have 
\begin{align*}
\widetilde \omega_{(p, q)}(\iota_p(\mu), -\mu(q) )&= \Phi(p)_q(A_p(\iota_p(\mu))(q) - \mu(q))\\
                                &= \Phi(p)_q(\mu(q) - \mu(q) )\\
                                & = 0.
                                \end{align*}
Next we show that $\widetilde\omega$ is invariant under the $\Aut(Q)$-action. Let $\psi \in \Aut(Q)$ then 
\begin{align*}
(R_\psi^* \widetilde\omega)_{(p, q)}(\xi, \zeta) &= \widetilde\omega_{(p \psi, \psi^{-1}(q))}\big((R_\psi)_*(\xi), (\psi^{-1}_{\phantom{-}*})_q(\zeta)\big)\\
&= \Phi(p\psi)_{\psi^{-1}(q)}\big( (A_p(R_\psi(\xi)) )(\psi^{-1}(q))  + (\psi^{-1}_{\phantom{-}*})_q(\zeta)\big)\\
&= \Phi(p\psi)_{\psi^{-1}(q)}\big( \ad(\psi^{-1})(A_p(\xi) ) (\psi^{-1}(q))+ (\psi^{-1}_{\phantom{-}*})_q(\zeta)\big)\\
&= (\psi^*\Phi(p))_{\psi^{-1}(q)}\big( (\psi^{-1}_{\phantom{-}*})_q( A_p(\xi)(q) + \zeta )\big)\\
&= \Phi(p)_q \big( (\psi_{*})_q (\psi^{-1}_{\phantom{-}*})_q( A_p(\xi)(q) + \zeta )\big)\\
&= \Phi(p)_q( A_p(\xi)(q) + \zeta) \\
&= \widetilde\omega_{(p, q)}(\xi, \zeta)
\end{align*}
as required.

It follows that $\widetilde\omega$ descends to a $\fg$-valued one-form on $\widetilde P$ which we denote by $\widetilde A$. 
We need to check that this is a connection one-form. Notice that $G$ acts on $P \times Q$ by acting on $Q$ 
and that this commutes with the action of $\Aut(Q)$ covering the action on $\widetilde P$.  If $\iota_{[p, q]}(\chi)$ is a vertical vector 
in $\widetilde P$ it lifts to $(0, \iota_q(\chi))$ at $(p, q)$.  Applying $\widetilde\omega_{(p, q)}$ we have
$$
\widetilde\omega_{(p, q)}(0, \iota_q(\chi)) = \Phi(p)_q( 0 + \iota_q(\chi)) = \chi.
$$
Hence $\widetilde A(\iota(\chi)) = \chi$.      Let $g \in G$, then we have 
\begin{align*}
  (R_g^* \widetilde\omega)  (p, q)(\xi, \zeta) &= \widetilde\omega_{(p, qg)}(\xi, (R_g)_*(\zeta))\\
  &= \Phi(p)_{qg}\big( A_p(\xi)(qg) + (R_g)_*(\zeta)\big)\\
  &= \Phi(p)_{qg}\big( (R_g)_*( A_p(\xi)(q) + \zeta)\big)\\
  &= \ad(g^{-1})\Phi(p)_q(A_p(\xi)(q) + \zeta)\\
 &= \ad(g^{-1})\widetilde\omega_{(p, q)}(\xi, \zeta)
  \end{align*}               
  so that $R_g^* \widetilde A = \ad(g^{-1}) \widetilde A$.  Here we use the fact that $A_p(\xi ) \in \Gamma_G(Q, TQ)$,
  so it satisfies $A_p(\xi)(qg) = (R_g)_*(A_p(\xi)(q))$. 
\end{proof}

The final thing we need to do is to show that when we apply the caloron transform twice in either direction the 
connections map to each other under the isomorphisms introduced in Section \ref{sec:unframed-fibrations}. 

Consider first the 
case that we start with a $G$-bundle $\widetilde P \to Y$ with a connection $\widetilde A$. The isomorphism in question  
goes from $\eta^*(\Eq(Q, \widetilde P)) \times_{\Aut(Q)} Q$ to $\widetilde P$ and is given by $(f, q) \mapsto f(q)$.  Under this isomorphism $\widetilde A$ pulls back to a connection on $\eta^*(\Eq(Q, \widetilde P)) \times_{\Aut(Q)} Q$ which we further pullback to 
$\eta^*(\Eq(Q, \widetilde P)) \times  Q$.  We note that 
$$
T_{(f, q)} \eta^*(\Eq(Q, \widetilde P)) \times Q  = \Gamma_G(Q, f^{*}(\widetilde P)) \times T_qQ 
$$
and the pullback of $\widetilde A$ applied to a pair $(\rho, \zeta)$ is 
$$
\widetilde A_{f(q)}(\rho(q) + f_*(\zeta)) = \widetilde A_{(f(q)}(\rho(q)) + \Phi(f)_q(\zeta).
$$
Consider now the connection constructed on $\eta^*(\Eq(Q, \widetilde P)) \times  Q$ by applying the constructions above twice.  First the connection $\widetilde A$ 
gives rise to a connection 
$$
A = \eta^* \Eq(Q, \widetilde A) + \Phi^r(\pi^*a)
$$
and the Higgs field identified in Proposition \ref{prop:higgs}.

The pair $(A, \Phi)$ gives rise to a $\fg$-valued one-form on $\eta^*(\Eq(Q, \widetilde P)) \times Q$ given by 
\begin{align*}
\widetilde\omega_{(f, q)}(\rho, \zeta) &= \Phi(f)_q\big( A_f(\rho)(q) + \zeta\big)\\
  &= \Phi(f)_q \big( \eta^*\Eq(Q, \widetilde A)_f(\rho)(q) + \Phi(a(\pi_*(\rho))) + \zeta \big)\\
  &= \Phi(f)_q\big(\widetilde A_{f(q)}(\rho(q))\big) +  \Phi(f)_q\big(\Phi(a(\pi_*(\rho)))\big) + \Phi(f)_q(\zeta).\\
  \end{align*}
Notice that under the identifications we are using (see Remark \ref{rem:inclusion}) $\widetilde A_{f(q)}(\xi(q))$ is a vertical vector 
in $Q$ and already identified with a vector in $\fg$, so there is no need to apply $\Phi(f)_q$.   Further the
vector $ \Phi(a(\pi_*(\rho)))$ is horizontal for $\widetilde A$, so applying $\Phi(f) = f^{*}(\widetilde A)$
gives zero. Hence we have
\begin{align*}
\widetilde\omega_{(f, q)}(\rho, \zeta) &= \widetilde A_{f(q)}(\rho(q)) + \Phi(f)_q( \zeta)\\
  \end{align*}
 as required.
 
Consider the second case where we start with an $\Aut(Q)$-bundle $P \to M$ with connection $A$ and 
Higgs field $\Phi$.  From these we construct a $G$-bundle 
$$
P \times_{\Aut(Q)} Q \to P \times_{\Aut(Q)} X
$$
with connection $\widetilde A$ whose pullback to $P \times Q$ we have called $\widetilde\omega$ defined
by (\ref{E:widetildeomega}):
$$
\widetilde \omega_{(p, q)}(\xi, \zeta) = \Phi(p)_q \left( A_p(\xi)(q) + \zeta \right).
$$
From this we construct the $\Aut(Q)$-bundle $\eta^*(\Eq(Q, P \times_{\Aut(Q)} Q)) \to M$
with connection $A'$ and Higgs field $(\Phi')^r$ where
$$
A' = \eta^*(\Eq(Q, \widetilde A)) + (\Phi')^r(\pi^*a).
$$
 Each $p \in P$ defines a 
natural map $f_p \colon Q \to P \times_{\Aut(Q)} Q$
by $f_p(q) = [p, q]$ and $\chi(p) =  f_p$ is the isomorphism of $\Aut(Q)$-bundles defined in 
Section \ref{sec:unframed-fibrations}:
$$
\chi \colon P \to \eta^*(\Eq(Q, P \times_{\Aut(Q)} Q)).
$$
We want to show that under this isomorphism $A'$ and $\Phi'$ pullback to $A$ and $\Phi$.  

First we consider $\Phi'$. 
From Proposition \ref{prop:higgs} we have that at $p$ the pullback of $\Phi'$ is $f_p^{*}(\widetilde A)$. 
 The frame $f_p \colon Q \to P \times_{\Aut(Q)} Q$ lifts to 
$Q \to P \times Q$ sending $q \mapsto (p, q)$ and the tangent to this is $\zeta \mapsto (0, \zeta)$.  Applying $\widetilde \omega$ gives $\Phi(p)_q(\zeta)$ as required. 

To calculate the pullback of $A'$ by $\chi$ we need to consider the 
tangent map to $\chi$ applied to $\xi \in T_p P$.  
The result is a section in $\Gamma_G(Q, f_p^{*}(P \times_{\Aut(Q)} Q))$
which we can lift to $P \times Q$ and realise as $q \mapsto (\xi, 0) \in T_pP \times T_qQ$.  Considering the first
term in $A'$ we have 
$$
\eta^*(\Eq(Q, \widetilde A))_{(p, q)}(\xi, 0) = \Phi(p)_q(A_p(\xi)(q) + 0)
$$
so that 
$$
A'_p(\xi) =  \Phi(p)(A_p(\xi) ) + \Phi^r(p)(a(\pi_*(\xi, 0))).
$$
 But
the  connection $a$ is the projection of the connection $A$ so we must have
$$
A'_p(\xi) =  \Phi(p)(A_p(\xi) ) + \Phi^r(p)(\b(A_p(\xi)).
$$
The relation in equation \eqref{eq:leftright} then tells us that
$$
A'_p(\xi) = A_p(\xi)
$$
as required.

\subsection{The unframed caloron correspondence for products with connections and Higgs fields}
We consider Theorem \ref{thm:fib-unframed-conn-higgs} and show how to reduce it in this case. 
Let $\widetilde P \to M \times X$ be a $G$-bundle of type $Q \to X$ with connection $\widetilde A$. Take
 as $a$,  the flat connection on $X \times M \to M$. We have seen that the 
 $\Aut( Q)$-bundle can be reduced to a $\cG$-bundle $P \to M$, whose fibre at $m \in M$ is all diffeomorphisms $f \colon Q \to \widetilde P_m$ covering $X \to X \times M$ given by $x \mapsto (x, m)$.  Therefore, in the notation of Theorem \ref{thm:fib-unframed-conn-higgs}, $\bar f (x) = (x, m)$.
 Let $\rho $ be  a tangent vector at $f$. Then $\bar \rho(x) = (0, \nu) \in T_xX \times T_m M$ for some constant
 vector $\nu \in T_mM$ and thus $v_a(\bar\rho(x)) = 0$ and we have the following reduction of Theorem
 \ref{thm:fib-unframed-conn-higgs}.

\begin{theorem}
\label{thm:prod-unframed-conn-higgs}
Let  $\widetilde P \to X \times M $ be a $G$-bundle  of type $Q \to X$  with connection $\widetilde A$. Then the $\cG$-bundle $P \to M$ has connection $A$ and Higgs field $\Phi$ defined as follows. 
 \begin{itemize}
\item Let $f \in P$ so that $f \colon Q \to \widetilde P_m$ for some $m \in M$ covering $x \mapsto (x, m)$, then 
 $\Phi(f) = f^*(\widetilde A)$;
\item 
 If $\rho \in  \Gamma(Q, f^*(T\widetilde P))$, then $A_f(\rho) \in \Gamma_G(Q, TQ)$ is given by 
$$
A_f(\rho)(q) = \widetilde A_{f(q)}(\rho(q)) .
$$

\end{itemize}
\end{theorem}

Going in the other direction we have a $\cG$-bundle $P \to M$ with connection $A$ and Higgs field $\Phi$ and define $\widetilde P = P \times_{\cG} Q \to M \times X$. In this case  $A$ is a $\Lie(\cG)$-valued one-form on $P$
so that if $\xi \in T_pP$, then $A_p(\xi) \in  \Gamma_G(Q, \fg)$ and applying the Higgs field 
has no effect. The formula in the fibration case therefore reduces to 
$$
\widetilde \omega_{(p, q)}(\xi, \zeta) =  A_p(\xi)(q) + \Phi(p)_q (\zeta)
$$
and we have the following Theorem.

\begin{theorem} 
\label{thm:prod-unframed-conn-form}
The one-form 
$$
\widetilde \omega_{(p, q)}(\xi, \zeta) =  A_p(\xi)(q) + \Phi(p)_q (\zeta)
$$
descends to a connection one-form $\widetilde A$ on the $G$-bundle $\widetilde P \to X \times M$.
\end{theorem}
 
\subsection{The framed caloron correspondences with connections and Higgs fields}  We leave it to the interested
reader to show that in the case of framings the connections we have defined already respect the framings and 
the Higgs fields take values in the appropriate space of framed connections.

\section{Characteristic classes for gauge group bundles}
\label{sec:classes}
In this section we shall use the caloron correspondence to calculate characteristic classes for $\cG$-bundles, following a similar approach as in \cite{Murray-Vozzo1}. We begin with a brief review of characteristic classes and Chern--Weil theory for $G$-bundles for the convenience of the reader. 

\subsection{Review of Chern--Weil theory}

Let $EG\to BG$ denote the universal bundle, with the property that any principal $G$-bundle over $M$ is isomorphic to the pullback of $EG \to BG$ by a classifying map $f\colon M\to BG$. Up to homotopy equivalence, the universal bundle is fully characterised by the fact that it is a principal $G$-bundle and $EG$ is a contractible space. A characteristic class for a $G$-bundle over $M$ is a class in $H^*(M)$ obtained by pulling back elements in the cohomology group $H^*(BG)$. Chern--Weil theory provides a mechanism for producing characteristic classes in de Rham cohomology. Let $I^k(\mathfrak g)$ denote the algebra of multilinear, symmetric, ad-invariant functions on $k$ copies of the Lie algebra $\mathfrak g$. Elements of $I^k(\mathfrak g)$ are called invariant polynomials. The Chern--Weil homomorphism is a map
\[cw\colon I^k(\mathfrak g) \to H^{2k}(M)\]
defined by $f\mapsto cw_f(A) = f(F,\dots, F)$, for $A$ a connection on the $G$-bundle with curvature $F$. The class $cw_f(A) \in H^{2k}(M)$ is independent of the choice of connection and represents a characteristic class of the bundle. For compact Lie groups $G$, the Chern--Weil homomorphism applied to the universal bundle is an isomorphism, which extends to an algebra isomorphism $cw\colon I^*(\mathfrak g)\xrightarrow{\sim} H^*(BG)$. The proof of these results can be found in many standard references such as \cite{Dupont}.

Since $f$ is multilinear and symmetric, we will adopt the convention that whenever $f$ has repeated entries they will be collected into one slot and written as a power, for instance $f(\underbrace{A, \dots , A}_{k}, \underbrace{B ,\dots, B}_l) = f(A^kB^l)$. 

\subsection{The curvature of the caloron connection}
\label{curvature}
The first step towards calculating characteristic classes is to determine the curvature of the caloron connection given in Theorem \ref{thm:prod-unframed-conn-form}. Recall that $\tilde A$ is expressed in terms of a connection $A$ and Higgs field $\Phi$ as follows. For a tangent vector $ (\xi, \zeta) \in T_{[p,q]} \wt P $ we have
$$
\tilde A_{(p,q)}(\xi, \zeta) = A_p(\xi)(q) + \Phi(p)_q (\zeta)  .
$$
The curvature $\tilde F$ of $\tilde \omega$ is given by the formula
$$
\tilde F = d \tilde A + \tfrac12 [\tilde A, \tilde A]  .
$$
For a pair of tangent vectors $V_1 = (\xi_1, \zeta_1), V_2 = (\xi_2, \zeta_2) \in T_{[p,q]} \wt P$, the commutator term is given by
\begin{align*}
\tfrac12 [\tilde A_{(p,q)}(V_1), \tilde A_{(p,q)} (V_2)]
	& = \tfrac12 [A_p (\xi_1) , A_p(\xi_2) ](q) + \tfrac12 [A_p(\xi_1)(q), \Phi(p)_q (\xi_2) ]\\ 
	&\hspace{1 cm} + \tfrac12 [ \Phi(p)_q (\zeta_1), A_p(\zeta_2)(q) ] + \tfrac12 [\Phi(p)_q (\zeta_1), \Phi(p)_q (\zeta_2)]\\
	& = \tfrac12 [A, A]_p (\xi_1,\xi_2)(q) + \tfrac12 [\Phi, \Phi] (p)_q (\zeta_1, \zeta_2) + [A, \Phi]_{(p,q)}(V_1, V_2) ,
\end{align*}
while the differential
$$
d \tilde A (V_1, V_2) = \frac12 \big( V_1 ( \tilde A(V_2) ) - V_2 ( \tilde A(V_1) ) - \tilde A ([V_1, V_2]) \big)
$$
can be expressed as a sum of four terms,
$$
d \tilde A = d_P A + d_Q \Phi + d_P \Phi  + d_Q A  .
$$
Here we are considering $A$ and $\Phi$ variously as forms on $P$ and $Q$ and as maps from these spaces. More specifically, we have
\\
\begin{itemize}
\item $d_P A$ is the derivative of $A$ considered as a 1-form on $P$:
\vspace{0.5ex}
$$
d_PA_{(p,q)}(V_1,V_2) = \frac 12 \Big((\xi_1A)_p(\xi_2)-(\xi_2A)_p(\xi_1)-A_p([\xi_1,\xi_2])\Big)(q)
$$\vspace{0.5ex}
\item $d_Q \Phi$ is the derivative of the Higgs field considered as a 1-form on $Q$, i.e. an element of $\cA$:
\vspace{0.5ex}
$$
d_Q\Phi_{(p,q)}(V_1,V_2) = \frac 12 \Big(\zeta_1(\Phi(p))_q(\zeta_2)-\zeta_2(\Phi(p))_q(\zeta_1)-\Phi(p)_q([\zeta_1,\zeta_2])\Big)
$$
\vspace{0.5ex}
\item $d_P \Phi$ is the derivative of the Higgs field considered as a map $P \to \cA$:
\vspace{0.5ex}
$$
d_P\Phi_{(p,q)}(V_1,V_2) = \frac 12 \Big((\xi_1\Phi)(p)_q(\zeta_2)-(\xi_2\Phi)(p)(\zeta_1)\Big)
$$
\vspace{0.5ex}
\item $d_Q A$ is the derivative of $A$ considered as an element of $\Lie (\cG)$, i.e.~a $G$-equivariant map $Q \to \fg$:
\vspace{0.5ex}
$$
d_QA_{(p,q)}(V_1,V_2) = \frac 12 \Big(\zeta_1(A_p(\xi_2))-\zeta_2(A_p(\xi_1))-A_p([\xi_1,\xi_2])\Big)(q)
$$\vspace{0.5ex}
\end{itemize}
Putting this all together, we have
$$
\tilde F = F_A + F_ \Phi + \nabla \Phi  ,
$$
where $F_A$ is the curvature of $A$, $F_\Phi$ is the curvature of $\Phi$ considered as a connection on $Q$ and $\nabla \Phi = d_P \Phi + [A, \Phi] + d_Q A$.

\begin{remark}
Comparing this with the caloron correspondence for loop groups \cite{Murray-Vozzo1}, the Higgs field there is a flat connection on the trivial bundle over the circle and hence $F_\Phi = 0$.
\end{remark}

\subsection{Caloron classes}

We can now proceed to define characteristic classes for the $\cG$-bundle $P \to M$ using the caloron correspondence.
Let $f$ be an invariant polynomial in $I^k(\mathfrak g)$. The Chern--Weil homomorphism for the caloron transform $\wt P \to M \times X$ determines a $2k$-form representing a class $cw_f(\wt A) \in H^{2k}(M \times X)$. Integrating this form over the $d$-dimensional manifold $X$ yields a closed $(2k-d)$-form on $M$, which we call the \emph{caloron class} of $P$. In short, we have
$$
I^k(\mathfrak g) \xrightarrow{cw_f(\wt A)} H^{2k} (M \times X) \xrightarrow{\int_X} H^{2k-d}(M)  ,
$$
and we denote by $\varsigma_{2k-d} (P)$ the resulting caloron class in $H^{2k-d}(M)$.

Note that in order to define a degree $r$ characteristic class for $P$, we must pick $k = (d + r)/2$, which in particular requires $r$ and $d$ to have the same parity\footnote{In the loop group case $X$ was the circle, so $\dim(X) = 1$. Hence the string classes are all of odd degree.}.
An explicit formula for the caloron classes of $P$ is obtained by calculating $\int_X cw(\wt A)$. Given a connection $A$ and Higgs field $\Phi$ for $P$, we know from Section \ref{curvature} that the curvature $\tilde F$ of the caloron connection $\tilde A$ on $\wt P$ splits into a sum of three terms,
$$
\tilde F = F_A + F_ \Phi + \nabla \Phi  .
$$
Inserted as arguments into an invariant polynomial $f \in I^k(\mathfrak g)$, we have 
\begin{align*}
cw_f(\wt A)	= f(\tilde F^k) = f\big((F_A + F_ \Phi + \nabla \Phi)^k\big)  .
\end{align*}
Note that the base is a product and the forms $F_A, F_\Phi$ and $\nabla_A \Phi$ are of type $(2,0), (0,2)$ and $(1,1)$ respectively. Integrating over $X$ picks out only the terms of the type $(2k-d, d)$. Hence, we have
$$
\varsigma_{2k-d} (P) = \int_X f\big((F_A + F_ \Phi + \nabla \Phi)^k\big)_{[2k-d,d]}
$$
which for the lowest values of the pair $(d,k)$ yields the following table:
\begin{center}
\begin{table}
 { \setlength{\extrarowheight}{1.5pt} \begin{tabular}{|l|l|l|l|l|}
\hline \backslashbox{$d$\hspace*{-5ex}}{\hspace*{-1.5ex}$k$} & $1$&$2$& $3$ \\\hline 
1 &$ \int_X f(\nabla \Phi)$ & $ \int_X f\Big(2F_A\nabla \Phi\Big)$ & $\int_X f\Big(3F_A^2\nabla \Phi\Big)$\\
2 & $ \int_X f(F_\Phi)$& $\int_X f \Big(\nabla \Phi^2 +2F_AF_\Phi \Big)$ & $ \int_X f \Big(3F_A\nabla \Phi^2+3F_A^2F_\Phi\Big)$\\
3 & - & $ \int_X f \Big(2\nabla \Phi F_\Phi \Big)$ & $\int_X f\Big(\nabla \Phi^3 + 3F_A\nabla \Phi F_\Phi + 3F_A F_\Phi \nabla \Phi\Big)$\\
4 & - & $ \int_X f \Big(F_\Phi^2 \Big)$&$ \int_X f \Big(3\nabla \Phi^2F_\Phi+3F_A F_\Phi^2\Big)$ \\
5 & - & - & $ \int_X f \Big(3\nabla \Phi F_\Phi^2\Big)$\\ 
6 & - & - & $ \int_X f \Big(F_\Phi^3\Big)$\\ \hline
\end{tabular} }
\caption{}
\label{table:classes}
\end{table}
\end{center}
\vspace{2.5ex}
In particular, for $X= S^1$,  we recover the string classes introduced in \cite{Murray:2003, Murray-Vozzo1},
$$
\varsigma_{2k-1} (P) = k \int_{S^1} f\Big(F_A^{k-1}\nabla \Phi \Big)  .
$$
More generally for a manifold $X$ of dimension $d$, the lowest degree caloron classes are given by 
\begin{align*}
\varsigma_0 (P) &= \int_X f\Big(F_\Phi^{d/2}\Big)  ,\\
\varsigma_1 (P) &= \frac{d+1}{2} \int_X f \Big(\nabla \Phi F_\Phi^{\frac{d-1}2}\Big)  ,\\
\varsigma_{2} (P) &= \frac{d+2}{2} \int_X f \left(F_AF_\Phi^{d/2}+\frac12 \sum_{j=0}^{\frac{d-2}{2}}\nabla \Phi F_\Phi^{\frac{d-2}{2}-j} \nabla \Phi F_\Phi^j\right)  ,\\
\varsigma_{3} (P) &= \frac{d+3}{2} \int_X f\left(\sum_{j=0}^{\frac{d-1}{2}} F_AF_\Phi^{\frac{d-1}{2}-j}\nabla \Phi F_\Phi^j+ \frac 13\sum_{j=0}^{\frac{d-3}{2}}\sum_{l=0}^j\nabla \Phi F_\Phi^{\frac{d-3}{2}-j} \nabla \Phi F_\Phi^{j-l}\nabla \Phi F_\Phi^l\right)  ,\\
\varsigma_{4} (P) &= \frac{d+4}{2} \int_X f\Bigg(\frac 12\sum_{j=0}^{\frac{d}{2}} F_AF_\Phi^{\frac{d}{2}-j} F_A F_\Phi^j+ \sum_{j=0}^{\frac{d-2}{2}}\sum_{l=0}^jF_A F_\Phi^{\frac{d-2}{2}-j} \nabla \Phi F_\Phi^{j-l}\nabla \Phi F_\Phi^l \\
&\hspace{3.5 cm} + \frac 14\sum_{j=0}^{\frac{d-4}{2}}\sum_{l=0}^j\sum_{m=0}^l\nabla \Phi F_\Phi^{\frac{d-4}{2}-j} \nabla \Phi F_\Phi^{j-l}\nabla \Phi F_\Phi^{l-m}\nabla \Phi F_\Phi^m\Bigg) . \\
\end{align*}

\noindent The formulas  get increasingly more complicated, involving more nested sums, as one goes to higher degrees in cohomology. However, in the special case when $G$ is abelian a straightforward calculation leads to the following general formula:
$$
\varsigma_{2k-d}(P) = \sum_{i=\frac d2}^{\min\{d, k\}} \binom{k}{i} \binom{i}{d-i} \int_X  f\Big(F_A^{k-i} F_\Phi^{d-i} (\nabla_A \Phi)^{2i-d}\Big).
$$
\section{Universal caloron classes}
\label{sec:universal} 

In this final section we restrict our attention to the framed and product case of the constructions in Section \ref{sec:description}.
There is  a natural model for the universal $\cG_0$-bundle which allows for explicit expressions for the caloron classes.
  

The group of based gauge transformations $\cG_0$ acts freely on the space of connections $\cA$ and the quotient is a smooth tame Fr\'echet manifold \cite{ACM}, which can be identified with the classifying space $B\cG_0$.
We want to calculate the caloron classes explicitly for the universal bundle $ \cA \to \cA / \cG_0$. We do this by choosing a connection and Higgs field; the connection is the  one described in \cite{Atiyah:1984} whose  horizontal subspaces are given by $\ker d_\omega^*$ and the Higgs field is simply $ \Phi \colon  \cA \xrightarrow{\id} \cA$. The operator $d_\omega^*$ is the adjoint of the covariant exterior differential $d_\omega$ with respect to the inner product 
\[(\alpha ,\beta) = \int_X \alpha \wedge * \beta\]
for $\alpha ,\beta \in \Omega^p(X,\ad Q)$. Because $\cA$ is an affine space, the  tangent space to $\cA$ at $\omega $ is $\Omega^1(X,\ad Q)$. The map into the vertical tangent space is 
$$
d_ \omega \colon \Gamma_0(X, \ad Q) \to T_ \omega\cA = \Omega^1(X,\ad Q)
$$
and we have a splitting $T_\omega\cA = \im d_\omega \oplus \ker d_\omega^*$ into orthogonal subspaces. Since the gauge transformations are based, the Laplacian
$$
d^*_ \omega d_ \omega \colon \Gamma_0(X, \ad Q)\to\Gamma_0(X, \ad Q)
$$
is invertible and we denote by $G_ \omega = ( d^*_ \omega d_ \omega)^{-1}$ the corresponding Green's operator. The connection form can now be expressed as $G_\omega d_ \omega ^*$ and by Theorem \ref{thm:prod-unframed-conn-form} the caloron connection
on $\tilde\cA = (\cA \times Q)/\cG_0$ over $\cA/\cG_0 \times X$ is given by
$$
\tilde A_{(\omega, q)}(\xi, \zeta) = G_ \omega d_ \omega ^*(q) (\xi)+ \omega_q(\zeta)
$$ 
for a tangent vector $ (\xi, \zeta) \in T_{[\omega, q]} \wt \cA$, where $G_ \omega d_ \omega ^*(\xi)$ is interpreted as a $G$-equivariant map $Q\to\mathfrak g$. Note that this caloron connection is framed. Recall that the curvature of $\tilde A$ consists of three terms,
\[\tilde F = F_A + F_\Phi + \nabla \Phi .\]
Since the Higgs field maps a connection $\omega \in \cA$ onto itself, clearly $F_{\Phi}$ evaluated on a pair of tangent vectors $V_1=(\xi_1, \zeta_1)$ and $V_2=(\xi_2, \zeta_2)$ at $(\omega, q) \in \tilde \cA$ is given by the curvature $F_\omega(q)(\zeta_1, \zeta_2)$ of $\omega$. 

Next let us consider the term $\nabla \Phi$. Recall that as with any  curvature we only need to evaluate it on horizontal  vectors.  Indeed if $V_1$ and $V_2$ are horizontal, then $\tilde F(V_1,V_2) = d\tilde A(V_1,V_2)$ so the commutator term does not contribute. Let therefore $\zeta_1, \zeta_2 \in T_qQ$ be horizontal for $\omega$ and $\xi_1, \xi_2 \in \Omega^1(X,\ad Q)$ be horizontal (i.e. $d_\omega^* \xi = 0$). As $\cA$ is an affine space we can extend $\xi_1$ and $\xi_2$ to constant vector fields whose Lie bracket must vanish.  We also extend $\zeta_1$ and $\zeta_2$ to vector fields. From Section \ref{curvature} we have
\[\nabla \Phi_{(\omega, q)}(V_1,V_2) = d_P \Phi_{(\omega, q)}(V_1,V_2) + d_Q \big(G_ \omega d_ \omega ^*\big)(q) (V_1,V_2)  ,\]
where 
\[d_P \Phi_{(\omega, q)}(V_1,V_2) = -\frac 12 \omega_q([\zeta_1, \zeta_2] )\]
since $\Phi$ is the identity map, and 
\begin{align*}
d_Q \big(G_ \omega d_ \omega ^*\big)&(q)(V_1,V_2)  \\
&= \frac 12 \Big(\zeta_1 \big( q \mapsto G_\omega d_\omega^* (q) (\xi_2)\big) - \zeta_2 \big(q \mapsto  G_\omega d_\omega^* (q) (\xi_1) \big) \\
     &= \frac 12 \Bigg( t_0\Big(t \mapsto  \omega_q(\xi_2) + t \xi_1(\zeta_2)(q) \Big) - t_0\Big(t \mapsto  \omega_q(\xi_1) + t \xi_2(\zeta_1)(q) \Big)   \Bigg)\\
    &= \frac 12 \Big(\xi_1(\zeta_2)(q) - \xi_2(\zeta_1)(q)  \Big),\\
    \end{align*}
                  where we have abused the notation here by treating $\omega$ both as a constant and a variable.  Hence the $(1,1)$ component of the curvature applied to horizontal vector fields $V_1$ and $V_2$ is 
              \[\nabla \Phi_{(\omega, q)}(V_1,V_2) = \frac 12 \Big(\xi_1(\zeta_2)(q) - \xi_2(\zeta_1)(q) )  - \omega_q([\zeta_1, \zeta_2] )\Big).\]     
   Finally we calculate
   $$
F_{A(\omega,q) }(V_1, V_2) = \frac 12 \Big(\xi_1(\omega \mapsto G_\omega d_\omega^*)(q) (\xi_2)  - \xi_2(\omega \mapsto G_\omega d_\omega^*)(q) (\xi_1) \Big).
$$
Consider the first term which is 
\begin{align*}
t_0\Big( t \mapsto G_{\omega + t\xi_1} d_{\omega+t\xi_1}^* (\xi_2)\Big)
&=t_0\Big( t \mapsto G_{\omega + t\xi_1}\Big)  d_{\omega}^* (\xi_2)  
+ G_{\omega}t_0\Big( t \mapsto  d_{\omega+t\xi_1}^* (\xi_2)\Big)\\
&= t_0\Big( t \mapsto G_{\omega + t\xi_1}\Big)0  + G_\omega t_0\Big(t \mapsto  d_\omega^* (\xi_2) + t \ad_{\xi_1}^*(\xi_2)\Big)\\
& = G_\omega \ad_{\xi_1}^*(\xi_2).
\end{align*}
Taking the adjoint is linear and therefore commutes with differentiation. Finally we note that $\ad_{\xi_1}^*(\xi_2) = - \ad_{\xi_2}^*(\xi_1)$ because
$$
\int_X \<\ad_{\xi_1}^*(\xi_2), \xi_3\> \vol_X  = \int_X \< \xi_2, \ad_{\xi_1}( \xi_3) \> \vol_X  =  - \int_X \< \ad_{\xi_2}^*(\xi_1), \xi_3 \>\vol_X
$$
and the inner product on $\fg$ is invariant. Hence we conclude that the curvature $F_A$ applied to  $V_1$ and $V_2$ is $G_\omega \ad_{\xi_1}^*(\xi_2) $.  
Putting this all together we have
\begin{multline*}
\tilde F_{(\omega,q)}(V_1,V_2) =  \\
G_\omega \ad_{\xi_1}^*(\xi_2) + F_\omega(q)(\zeta_1, \zeta_2) +  \frac 12 \Big(\xi_1(\zeta_2)(q) - \xi_2(\zeta_1)(q)  - \omega_q([\zeta_1, \zeta_2] )\Big).
\end{multline*}

Using this $\tilde F$ in the formulas in Table \ref{table:classes} for $d=4$ and $k=3$ we reproduce the results of \cite{Atiyah:1984} 
and for $d=k=3$ the  results of \cite{CarMicMur97Index-theory-gerbes}.



\begin{thebibliography}{99}

\bibitem{ACM} 
M. C. Abbati, R. Cirelli and A. Mani\`a.
The orbit space of the action of gauge transformation group on connections.
\emph{J. Geom. Phys.}, 
{\bf 6}(4), 537--557, 1989.

\bibitem{Ati}
M.F. Atiyah.
Complex analytic connections in fibre bundles.
\emph{Trans. Amer. Math. Soc.},
{\bf 85}, 181--207, 1957.

\bibitem{Atiyah:1984}
M. F. Atiyah and I. M. Singer.
Dirac operators coupled to vector potentials.
\emph{Proc. Nat. Acad. Sci. U.S.A.},
{\bf 81}(8, Phys. Sci.), 2597Ð-2600, 1984.

\bibitem{Bergman:2005}
A. Bergman and U. Varadarajan.
Loop groups, Kaluza-Klein reduction and M-theory.
\emph{J. High Energy Phys.},
(6), 043, 28 pp. (electronic), 2005.

\bibitem{Bouwknegt:2009}
P. Bouwknegt and V. Mathai.
T-duality as a duality of loop group bundles.
\emph{J. Phys. A: Math Theory},
{\bf 42}(16), 2009.


\bibitem{CarMicMur97Index-theory-gerbes}
A.~Carey, J.~Mickelsson, and M.~Murray.
\newblock Index theory, gerbes, and {H}amiltonian quantization.
\newblock {\em Comm. Math. Phys.}, {\bf 183}(3), 707--722, 1997.


\bibitem{Cur}
W.D. Curtis.
The automorphism group of a compact group action.
\emph{Trans. Amer. Math. Soc.},
{\bf 203},  45--54, 1975.

\bibitem{Dupont}
J. L. Dupont. 
\emph{Curvature and characteristic classes.} 
Lecture Notes in Mathematics, Vol. 640.
Springer--Verlag, Berlin, 1978.

\bibitem{Garland:1988}H. Garland and M. K. Murray. 
Kac-Moody monopoles and periodic instantons. 
\emph{Comm. Math. Phys.}, {\bf 120}(2), 335--351, 1988.

\bibitem{Ham}
Richard S. Hamilton.
The inverse function theorem of {N}ash and {M}oser.
\emph{Bull. Amer. Math. Soc. (N.S.)},
 {\bf 7}(1),  65--222, 1982.

\bibitem{Kobayashi:1963}
S. Kobayashi and K. Nomizu.
\emph{Foundations of differential geometry. Vol 1.}
Interscience Publishers, New York--London, 1963.

\bibitem{Murray:2003}
M. K. Murray and D. Stevenson.
Higgs fields, bundle gerbes and string structures.
\emph{Commun. Math. Phys.}, {\bf 243} (3), 541--555, 2003.

\bibitem{Murray-Vozzo1}
M. K. Murray and R. F. Vozzo.
The caloron correspondence and higher string classes for loop groups.
\emph{J. Geom. Phys.}, {\bf 60}(9), 1235--1250, 2010.

\bibitem{Murray-Vozzo2} M. K. Murray and R. F. Vozzo.
Circle actions, central extensions and string structures.
\emph{Int. J. Geom. Methods Mod. Phys.},
{\bf 7}(6), 1065--1092, 2010.

\bibitem{Vozzo:2010} R. F. Vozzo.
Loop groups, string classes and equivariant cohomology.
\emph{J. Aust. Math. Soc.},
{\bf 90}(1), 2011.

\end{thebibliography}

\end{document}